\documentclass[reqno,a4paper]{amsart}
\usepackage{amsmath}
\usepackage{amsfonts}
\usepackage{amssymb}
\usepackage{latexsym}

\numberwithin{equation}{section}
\newtheorem{proposition}{Proposition}[section]
\newtheorem{corollary}[proposition]{Corollary}
\newtheorem{theorem}[proposition]{Theorem}
\newtheorem{conjecture}[proposition]{Conjecture}

\newtheorem{lemma}[proposition]{Lemma}

\newcommand{\sortable}[2]{\ensuremath{W_{#1,#2}}}
\newcommand{\exactly}[2]{\ensuremath{E_{#1,#2}}}

\newcommand{\id}{\mbox{id}}
\newcommand{\sym}{\mathfrak{S}}

\newcommand{\st}{\mbox{\ensuremath{\ast}}}
\newcommand{\q}{\mbox{\ensuremath{\scriptstyle{?}}}}

\newcommand{\glob}[1]{\ensuremath{\langle\, #1\,\rangle}}
\newcommand{\IFF}{if and only if}
\newcommand{\fpoo}{forbidden pattern of order}
\newcommand{\ufpoo}{uninterrupted forbidden pattern of order}

\DeclareMathOperator{\ssc}{\mathrm{ssc}}

\title[Permutations sortable by $n-4$ passes through a stack]
{Permutations sortable by \\ $n-4$ passes through a stack}
\author{Anders Claesson, Mark Dukes and Einar Steingr\'imsson}
\thanks{The first and third authors were supported by grant no.\
  060005013 from the Icelandic Research Fund.  The second author
  acknowledges funding by the EC's Research Training Network
  `Algebraic Combinatorics in Europe', grant HPRN-CT-2001-00272 while
  at Universit\'{e} Bordeaux 1, France.}

\address{The Mathematics Institute,
  Reyjav\'ik University, 103 Reykjav\'ik, Iceland.}
\address{Science Institute, University of Iceland, 107
  Reykjav\'ik, Iceland.}

\begin{document}
\begin{abstract}
  We characterise and enumerate permutations that are sortable by
  $n-4$ passes through a stack.  We conjecture the number of
  permutations sortable by $n-5$ passes, and also the form of a
  formula for the general case $n-k$, which involves a polynomial
  expression.
\end{abstract}

\maketitle
\thispagestyle{empty}

\section{Background}
We view permutations as words without repeated letters; if $\pi$ is a
permutation of an $n$ element set and $\pi(i)=\pi_i$, then we write
$\pi=\pi_1\dots \pi_n$. The stack sorting operator $S$ can be defined
recursively on permutations of finite subsets of $\{1, 2,\dots\}$ as
follows. If $\pi$ is empty then $S(\pi) = \pi$.  If $\pi$ is nonempty
write $\pi$ as the concatenation $\pi =LnR$, where $n$ is the
greatest element of $\pi$ and $L$ and $R$ are the subwords to the left
and right of $n$ respectively. Then
$$
S(\pi)=S(L)S(R)n.
$$ 
For example, $S(42513)=24135$.
We say that a permutation $\pi$ is $k$-\emph{stack sortable} if
$S^k(\pi)=\id$, where $S^k=S\circ S^{k-1}$, $S^0$ is the identity
operator and $\id$ is the identity permutation $12\dots n$.  Let the
\emph{(stack sorting) complexity} of $\pi$, denoted $\ssc\pi$, be the
smallest $k$ such that $\pi$ is $k$-stack sortable. Let $\sym_n$ be
the set of permutations of $\{1,\dots,n\}$. For the permutations in
$\sym_3$ we have
$$
\begin{array}{rc|c|c|c|c|c}
  \pi    \;\,=\!\!\! & 123 & 132 & 213 & 231 & 312 & 321 \\
  \ssc\pi\;\,=\!\!\! &  0  &  1  &  1  &  2  &  1  &  1 
\end{array}
$$

Let $\sortable{n}{k}$ be the set of all $k$-stack sortable
permutations in $\sym_n$; in other words, $\sortable{n}{k}$ is the
set of permutations in $\sym_n$ whose complexity is at most
$k$.  Let $\exactly{n}{k}$ be the set of permutations in
$\sym_n$ whose complexity is exactly $k$. Note that
\begin{align*}
  \exactly{n}{k}  &= \sortable{n}{k}-\sortable{n}{k-1}; \\
  \sortable{n}{k} &= \sym_n - (\exactly{n}{k+1}\cup\dots\cup\exactly{n}{n-1}).
\end{align*}

It is easy to see that $\sortable{n}{n-1}=\sym_n$.
Knuth~\cite[2.2.1.5]{knuth} leaves as an exercise to the reader 
to show that $\sortable{n}{1}=\sym_n(231)$.
In his PhD thesis West \cite{west.phd} showed that
$$
\sortable{n}{2}=\sym_n(2341,3\overline{5}241),
$$ 
where $\sym_n(2341,3\overline{5}241)$ is the set of permutations in
$\sym_n$ that avoid the pattern $2341$ and the ``barred'' pattern
$3\overline{5}241$; for information on these, see \cite{es-survey},
especially Section~7.  West also showed that $\sortable{n}{n-2}$ are
precisely those permutations that do not have suffix $n1$. This last
statement is easily shown by proving that the permutations in
$\exactly{n}{n-1}$ are those with suffix $n1$.

In addition, West characterized $\exactly{n}{n-2}$. To state that
result it is convenient to introduce some notation for special sets of
words over the alphabet $\{1,2,\dots\}$. Let an asterisk ($\st$) stand
for any word of zero or more characters, and let a question mark
($\q$) stand for any single letter. These conventions are adopted from
the so called glob patterns in computer science. For a word $w$ over
$\{\st,\q\}\cup\{1,2,\dots\}$, let $\glob{w}$ denote the set of
words of the form $w$. For instance, \glob{\st n1} consists of all
words with suffix $n1$, and $\sym_5\cap\glob{\st 51\q}=\{ 23514,
24513, 32514, 34512, 42513, 43512\}$.  Let
$$\glob{w_1,\dots,w_k}=\glob{w_1}\cup\dots\cup\glob{w_k}.
$$ 
West's characterisation of $\exactly{n}{n-1}$ and $\exactly{n}{n-2}$
can then be stated as in the following two lemmas, which follow from
Theorems 4.2.4 and 4.2.17 in \cite{west.phd} and their proofs.
\begin{lemma}\label{yada1} For all $n\geq 2$,
  $$\exactly{n}{n-1} = \sym_n\cap\glob{\st n1}.
  $$
Thus, the cardinality of $\exactly{n}{n-1}$ is  $(n-2)!$.
\end{lemma}
\begin{lemma}\label{yada2} For all $n\geq 4$,
  $$\exactly{n}{n-2} = \sym_n \cap 
    \glob{\;\st n2,\;\st(n-1)1n,\; \st n1\q,\; \st n\q 1,
      \; \st n\st (n-2)\st (n-1)1\;}.
    $$
Thus, the cardinality of $\exactly{n}{n-2}$ is  $(n-3)!(7n-12)/2$.
\end{lemma}
By subtracting the cardinalities in Lemmas \ref{yada1} and \ref{yada2}
from $n!$, we get the following result.
\begin{proposition} For all $n\geq 4$ the cardinality of $W_{n,n-3}$ is 
$$\dfrac{(n-3)!}{2}\left(2n^3-6n^2-5n+16\right).$$
\end{proposition}

\section{Permutations requiring exactly $(n-3)$-stack sorts}

\begin{theorem}\label{main}
  For all $n\geq 6$, the set of permutations $\exactly{n}{n-3}$ are
  those given in column labeled ``Type'' of Table~\ref{thelist}.
  The number of such permutations is
  \begin{eqnarray}
    \frac{(n-4)!}{3} \left( 47 \binom{n-6}{2} + 194 \binom{n-6}{1} +
    297 \right) \label{threenumbers}.
  \end{eqnarray}
\end{theorem}

To prove this we require some terminology and results from
West~\cite[\S 4.2]{west.phd}.  A \emph{forbidden pattern of order $k$}
in a permutation $\pi$ is a triple $(B,c,a)$, where $B$ is a
subsequence of length $k$ in $\pi$ and $(c,a)$ is a pair of entries in
$\pi$ such that for every $b\in B$ the subsequence $bca$ is an
occurrence of the pattern $231$, that is, $a<b<c$. In such a situation
we say that the pair $(c,a)$ \emph{witnesses} the forbidden pattern
$B$. We call a forbidden pattern $(B,c,a)$ \emph{uninterrupted} if
there is no subsequence $bxb'$ in $\pi$ where $b,b'\in B$ and $x>c$.

\begin{lemma}{\cite[Theorems 4.2.10 and 4.2.14]{west.phd}}\label{lemma}
  Let $\pi$ be a permutation. 
  \begin{itemize}
  \item[(i)] $\ssc(\pi)\leq k$ if $\pi$ does not contain any \fpoo\ $k$;
  \item[(ii)] $\ssc(\pi)> k$ if $\pi$ contains an \ufpoo\ $k$.
  \end{itemize}
\end{lemma}

\begin{table}
  $\begin{array}{l|c|l} \hline\hline
  &&\\[-1ex]
  \mbox{Case} &\mbox{Type} &\mbox{Number} \\[-1ex]
  &&\\\hline\hline
  \multicolumn{3}{l}{} \\
  \multicolumn{3}{l}{\pi_n=n} \\[1ex] \hline
  &&\\[-1ex]
  \mbox{1(a)} & \st (n-1)2n & (n-3)! \\[.8ex]
  \mbox{1(b)} & \st (n-1)\q 1n & (n-3)! \\[.8ex]
  \mbox{1(c)} & \st (n-1)1\q n & (n-3)! \\[.8ex]
  \mbox{1(d)} & \st (n-1) \st (n-3)\st (n-2)1n & (n-3)!/2 \\[.8ex]
  \mbox{1(e)} & \st (n-2)1(n-1)n  & (n-4)! \\[-1ex]
  &&\\\hline
  \multicolumn{3}{l}{}\\
  \multicolumn{3}{l}{\pi_{n-1}=n} \\[1ex] \hline
  &&\\[-1ex]
  \mbox{2(a)} & \st n3 & (n-2)! \\[0.8ex]
  \mbox{2(b)} & 
  \st (n-1)1n\q  &  (n-5)(n-4)! \\[0.8ex]
  \mbox{2(c)} & 
  \st (n-2)1n(n-1) &  (n-4)! \\[-1ex]
  &&\\\hline
  \multicolumn{3}{l}{}\\
  \multicolumn{3}{l}{\pi_{n-2}=n} \\[1ex] \hline
  &&\\[-1ex]
  \mbox{3(a)} & 
  \st n2\q & (n-3)(n-3)! \\[0.8ex]
  \mbox{3(b)} & 
  \st n\q 2 & (n-3)(n-3)! \\[-1ex]
  &&\\\hline
  \multicolumn{3}{l}{}\\
  \multicolumn{3}{l}{\pi_{n-3}=n} \\[1ex] \hline
  &&\\[-1ex]
  \mbox{4(a)} & 
  \st n\q \q 1, \mbox{ but not } \st n(n-2)(n-1)1 &  (n-2)!-(n-4)! \\[0.8ex]
  \mbox{4(b)} & \st n \q 1 \q & (n-2)! \\[0.8ex]
  \mbox{4(c)} & \st n 1\q \q  & (n-2)!  \\[0.8ex]
  \mbox{4(d)} & \st n(n-2)(n-1)2 & (n-4)! \\[-1ex]
  &&\\\hline
  \multicolumn{3}{l}{}\\
  \multicolumn{3}{l}{\pi_{n-i}=n \mbox{ and }i> 3}\\[1ex] \hline
  &&\\[-0.9ex]
  \mbox{5(a)} 
  & \st n A (n-2) B(n-1)2 , \mbox{ where $A\cup B\neq \emptyset$} 
  & (n-2)!/2-(n-4)! \\[0.8ex]
  \mbox{5(b)} & \st n\st (n-3)\st (n-1)(n-2)1 & (n-3)!/2 \\[0.8ex]
  \mbox{5(c)} & 
  \st n \st (n-1) \st (n-3) \st (n-2)1 & (n-2)!/6 \\[0.8ex]
  \mbox{5(d)} & \st (n-1) \st n \st\!\left\{\!\!\!
  \begin{array}{c} 
    (n-3) \st (n-4) \\
    (n-4) \st (n-3) 
  \end{array}
  \!\!\!\right\}\!\st (n-2)1 & (n-2)!/12 \\[2ex]
  \mbox{5(e)} & \st n \st (n-2) \st (n-1)\!\left\{\!\!\!
  \begin{array}{c} 
    1\q \\ \q 1
  \end{array}\!\!\!\right\} &  (n-4)(n-3)!  \\[2ex]
  \mbox{5(f)} &\st n\st (n-3)\st (n-1)1(n-2) &(n-3)!/2\\[0.8ex] 
  \mbox{5(g)} & \st n \st (n-3) \st (n-2)1(n-1)&(n-3)!/2 \\[0.8ex] 
  \mbox{5(h)} & \st (n-2)\st n\st\!\left\{\!\!\!
  \begin{array}{c} 
    (n-3)\st (n-4)\\
    (n-4)\st (n-3) 
  \end{array}\!\!\!\right\}\!\st (n-1)1 & (n-2)!/12 \\[-1ex]
  &&\\\hline
  \end{array}$\bigskip
  \caption{Permutations in $\exactly{n}{n-3}$}
  \label{thelist}
\end{table}

\begin{proof}[Proof of Theorem~\ref{main}]
  Let $\pi \in \sym_n$. Using the contrapositive of (i) and (ii) in
  Lemma~\ref{lemma}, if $\ssc(\pi)=n-3$, then $\pi$ contains a \fpoo\
  $n-4$ and does not contain an \ufpoo\ $n-3$. The \fpoo\ $n-4$ must
  be witnessed by entries that appear in some two of the positions
  $n-3$, $n-2$, $n-1$ or $n$ of the permutation.  We will first
  condition on the position of the largest entry in the permutations,
  then condition on those permutations that contain a forbidden
  pattern of the required order, and finally single out those
  permutations that are in \exactly{n}{n-3}.
  
  Suppose $\pi_n=n$. A permutation $\pi=\pi' n$ is a member of
  \exactly{n}{n-3} precisely when $\pi'$ is a member of
  $\exactly{n-1}{n-3}$. Thus entries 1(a)--1(e) of Table~\ref{thelist}
  immediately follow from Lemma~\ref{yada2}.

  Suppose $\pi_{n-1}=n$. We must have $\pi_n \geq 3$, for otherwise,
  by Lemmas~\ref{yada1} and~\ref{yada2}, we would have
  $\pi\in\exactly{n}{n-1}$ or $\pi\in\exactly{n}{n-2}$. If $\pi_n=3$
  then the permutation is in $\exactly{n}{n-3}$, and hence we get
  2(a). If $k=\pi_n>3$ then it is impossible for $(n,k)$ to witness a
  forbidden pattern $F$ of order $n-4$ since there could be at most
  $n-k-1<n-4$ elements in $F$. Thus the forbidden pattern $F$ must be
  witnessed by $(\pi_{n-3},\pi_{n-2})$. As there are now $n-2$ values
  from which to form the forbidden pattern $F$, we are forced to
  choose the extreme values from this set as the values for
  $(\pi_{n-3},\pi_{n-2})$. Consequently, if $k=n-1$ then we must have
  $(\pi_{n-3},\pi_{n-2})=(n-2,1)$, giving case 2(c).  Otherwise
  $3<k<n-1$ and $(\pi_{n-3},\pi_{n-2}) =(n-1,1)$ which gives case
  2(b). One easily verifies that all such permutations are in
  $\exactly{n}{n-3}$.

  Suppose $\pi_{n-2}=n$. By Lemma~\ref{yada2} the value 1 cannot be to
  the right of $n$. Let $k$ be the smaller of the two values
  $\pi_{n-1}$ and $\pi_{n}$. If $k>2$ then the pair $(n,k)$ witnesses
  a \fpoo\ at most $n-k-1-1\,<\,n-4$. For the same reason,
  $(\pi_{n-1},\pi_{n})$ cannot witness a \fpoo\ $n-4$. So exactly one
  of the entries to the right of $n$ must be 2, giving 3(a) and (b).
  One easily verifies that all such permutations are in
  $\exactly{n}{n-3}$.

  Suppose $\pi_{n-3}=n$. If 1 is to the right of $n$ in $\pi$ then
  $(n,1)$ witnesses a \fpoo\ $n-4$. 
  However, if $\pi \in \glob{\st n(n-2)(n-1)1}$, it has one of the forms
  given in Lemma~\ref{yada2}.
  In all other cases the
  permutation is in $\exactly{n}{n-3}$. This gives 4(a), (b) and (c).

  Alternatively, if 1 is to the left of $n$, let $k$ be the value of
  the smallest entry to the right of $n$. Since there are at most
  $n-5$ entries to the left of $n$ in $\pi$ that take values between
  $k$ and $n$, the pair $(n,k)$ cannot witness a \fpoo\ $n-4$.  In
  order for $\pi$ to contain a forbidden pattern $F$ of order $n-4$ it
  must be witnessed by the pair $(\pi_{n-1},\pi_{n})$ and have as the
  block of $n-4$ values the value $\pi_{n-2}$ along with the $n-5$
  elements to the left of $n$ that are not 1. Thus $\pi_n=2$,
  $\pi_{n-1}=n-1$ and
  $$\pi\in\glob{\st nk (n-1)2}
  $$ 
  for some $k\neq 1$.  One easily checks
  that the only value of $k$ for which $\pi \in \exactly{n}{n-3}$ is
  $k=n-2$. This gives 4(d).

  Suppose $\pi_{n-i}=n$ for some $i>3$. 
  Then $(n,k)$ can never witness a \fpoo\ $n-4$.
  In the remainder of the proof, we condition on the relative
  positions of $(n-2)$, $(n-1)$ and $n$. If $\pi = A n B (n-1) C (n-2)
  D$ where $|B\cup C \cup D|\geq 2$, then
  $$S(\pi) = S(A) S(B) S(C) S(D)(n-2)(n-1)n.$$
  This gives
  \begin{align*}
    {\pi\in \exactly{n}{n-3}}
    &\iff S(A)S(B) S(C) S(D)\in \exactly{n-3}{n-4} \\
    &\iff S(A)S(B)S(C)S(D) \in \glob{*(n-3)1}.
  \end{align*}
  Since $|B\cup C \cup D|\geq 2$, the ways in which this can happen
  are restricted to 

  \begin{itemize}
  \item[(i)] $D=1$ and $n-3 \in C$, giving case 5(c), 
  \item[(ii)]  $D=1$, $C=\emptyset$ and $n-3 \in B$, giving case 5(b),
  \item[(iii)] $D=\emptyset$, $C=1$ and $n-3 \in B$, giving case 5(f).
  \end{itemize}

\noindent
  If $\pi = A n B (n-2) C (n-1) D$ where $|B\cup C \cup D|\geq 2$,
  then 
  $$
  S(\pi) = S(A) S(B) S(C) (n-2) S(D) (n-1) n.
  $$ 
  This implies that $\pi$ belongs to $\exactly{n}{n-3}$ \IFF\ 
  $$S(A)S(B)S(C)(n-2)S(D) \in \exactly{n-2}{n-4}.
  $$ If $D=\emptyset$ then $\pi$ belongs to $\exactly{n}{n-3}$
  \IFF\ $S(A) S(B) S(C)$ belongs to $\exactly{n-3}{n-4}$, which
  happens \IFF\ $C=1$ and $(n-3) \in B$, from which we get 5(g).
  Otherwise $ S(A) S(B) S(C) (n-2) S(D)$
  belongs to the types listed in Lemma~\ref{yada2}. There are 3 cases to
  consider:
  \begin{itemize}
  \item[(i)] $S(D)=2$, so $\pi = A n B (n-2) C (n-1)2$ where
    $B\cup C \neq \emptyset$, giving case 5(a); 
  \item[(ii)] $S(D)=1k$, so $D=1k$ or $D=k1$, giving case 5(e); and
  \item[(iii)] it is not possible that $S(A) S(B) S(C) (n-2) S(D)$
    matches the last type in Lemma~\ref{yada2} since this would mean
    that $S(D)$ ends in 1, its smallest entry.
  \end{itemize}

  If $\pi = A (n-1) B n C (n-2) D$ where $|C\cup D|\geq 3$, then a
  \fpoo\ $n-4$ can only be witnessed by the pair $(n-2,1)$.
  Furthermore, all elements $2$ through $n-3$ must be to the left of
  $n-2$, so $\pi_{n-1}=n-2$ and $\pi_n=1$, that is $\pi = A(n-1) B n C
  (n-2)1$ where $|C|\geq 2$. From this we have $S^2(\pi) =
  S(S(A)S(B)) S(S(C)1) (n-2)(n-1)n$ and
  \begin{align*}
    \pi \in \exactly{n}{n-3} & \iff
    \pi' = S(S(A)S(B)) S(S(C)1) \in \exactly{n-3}{n-5}.
  \end{align*}
  The conditions on $\pi'$ (and therefore $\pi$) are easily derived by
  comparing $\pi '$ to the types in Lemma~\ref{yada2}. Since $|C|\geq
  2$ one cannot have $(n-3)2$ as a suffix of $S(S(C)1)$. Similarly,
  $(n-3)1k$ and $(n-3)k1$ cannot be suffixes of $S(S(C)1)$. Also, it
  is only possible that 
  $$\pi'\in\glob{\st (n-3)\st (n-5)\st (n-4)1}
  $$
  if $C=\emptyset$, which is not allowed. The only possibility for
  $\pi'$ is that it has $(n-4)1(n-3)$ as a suffix, and so $n-4,n-3 \in
  C$ since $|C|\geq 2$.  Under these conditions, $S^2(\pi)$ has suffix
  $(n-4)1(n-3)(n-2)(n-1)n$ and $\pi \in \exactly{n}{n-3}$.  Thus
  $D=1$ and $n-4,n-3 \in C$, from which we get case 5(d).

  If $\pi = A (n-2) B n C (n-1) D$ where $|C\cup D|\geq 3$ then
  $$S(\pi) = S(A) S(B) (n-2) S(C)S(D) (n-1)n.
  $$ Now, $\pi$ belongs to $\exactly{n}{n-3}$ \IFF\ $\pi' =
  S(A)S(B)(n-2)S(C)S(D)$ belongs to $\exactly{n-2}{n-4}$. Since
  $|C\cup D|$ is at least 3, the entry $n-2$ has at least 3 entries to
  its right in $\pi'$.  By comparing this to the possible types that
  it may take in Lemma~\ref{yada2}, we find that the only possibility
  is
  $$\pi' \in\glob{\st (n-2) \st (n-4)\st (n-3)1}.
  $$
  This happens \IFF\ $D=1$ and $n-4,n-3 \in C$, which gives case 5(h).

  For the two final cases in which $n-2$ and $n-1$ are to the left of
  $n$, it is not possible for $\pi$ to contain a \fpoo\ $n-4$, since
  at least one of $n-2$, $n-1$ and $n$ is needed in order to witness
  such a pattern, and none of them are in the rightmost four positions
  of the permutation.

  The number of permutations of each type is shown in column three of
  Table~\ref{thelist}. Adding these gives $(47n^2-223n+240)(n-4)!/6$
  which may be be rewritten as in formula (\ref{threenumbers}).
\end{proof}

\begin{corollary}
  For $n\geq 6$, the collection of $(n-4)$-stack sortable permutations
  in $\sym_n$ are those permutations that are not of the types listed
  in Lemma~\ref{yada1}, Lemma~\ref{yada2} or Table~\ref{thelist}. The
  number of these is $(n-4)!(3n^4-18n^3-4n^2+158n-192)/3$.
\end{corollary}

It is now straightforward to write down (a lengthy expression for) the
descent polynomial of the $(n-4)$-stack sortable permutations, that
is, the polynomial whose $k$-th coefficient is the number of
$(n-4)$-stack sortable permutations with exactly $k$ descents (a
descent in $\pi=a_1a_2\dots a_n$ is an $i$ such that $a_i>a_{i+1}$).

The following conjecture is based on computer generated data for
$n\le13$.

\begin{conjecture}
  For all $n\geq 8$,
  the number of permutations in $\exactly{n}{n-4}$ is
  \begin{equation*}
    \frac{(n-5)!}{10}
    \left( 
    854\binom{n-8}{3}+ 5099\binom{n-8}{2} + 12545\binom{n-8}{1} + 16130 
    \right), 
  \end{equation*}
  or equivalently, the number of $(n-5)$-stack sortable permutations is
  $$\frac{(n-5)!}{60}\left(
  60n^5 - 600n^4 + 506n^3 + 11241n^2 - 38369n + 34236
  \right).
  $$
\end{conjecture}

We end with a conjecture about the form of an expression for the
number of permutations needing exactly $n-k$ stack sorts.  This has
been verified for all $n\le14$ and all relevant $k$.

\begin{conjecture}
  For all $n\geq 2k$,
  the number of permutations in $\exactly{n}{n-k}$ may be written as
  \begin{equation*}
    \frac{(k-1)!(n-k-1)!}{\big(2(k-1)\big)!} 
    \sum_{i=0}^{k-1} a_i \binom{n-2k}{i} 
  \end{equation*}
  where $a_i \in \mathbb{N}$.
\end{conjecture}


\begin{thebibliography}{99}
\bibitem{knuth} D. E. Knuth, {\it The art of computer
    programming. Vol. 1: Fundamental algorithms}, Addison Wesley
  Publishing Co., Reading, Mass.-London-Don Mills, Ont, 1969.
\bibitem{es-survey} E.\ Steingr\'imsson: Generalized permutation
  patterns --- a short survey, ``Permutation Patterns, St
  Andrews~2007,'' S.A. Linton, N. Ruskuc, V.  Vatter (eds.), LMS
  Lecture Note Series, Cambridge University Press, to appear.

\bibitem{west.phd} J. West, Permutations with forbidden subsequences;
  and, Stack sortable permutations, Ph.D. thesis, Massachusetts
  Institute of Technology, 1990.
\end{thebibliography}
\end{document}